\pdfoutput=1
\documentclass[oneside,a4paper,12pt]{amsart}

\expandafter\let\csname ver@amsthm.sty\endcsname\relax
\let\theoremstyle\relax

\usepackage[utf8]{inputenc}
\usepackage[T1]{fontenc}
\usepackage[english]{babel}
\usepackage{amsmath, amsfonts, amssymb, color}
\usepackage{lmodern}
\usepackage{mathtools}
\usepackage{enumitem}
\usepackage{blkarray}
\usepackage{ytableau}
\usepackage{ifthen}
\usepackage[pdftex,hidelinks,bookmarks=true]{hyperref}
\usepackage{amsthm}
\usepackage[english]{cleveref}
\usepackage{comment}
\usepackage[babel]{csquotes}
\usepackage[left=3cm,right=3cm,bottom=3cm]{geometry}

\makeatletter
\let\xx@thm\@thm
\AtBeginDocument{\let\@thm\xx@thm}
\makeatother

\usepackage[scaled=0.86]{berasans}
\usepackage[sc,osf]{mathpazo}

\usepackage{todonotes}

\usepackage[all,2cell]{xy}  
\UseAllTwocells
\SelectTips{eu}{10}

\usepackage{setspace}
\setdisplayskipstretch{.6}

\newtheorem{lemma}{Lemma}
\newtheorem{theorem}    [lemma]{Theorem}

\newtheorem{proposition}[lemma]{Proposition}
\theoremstyle{definition}

\newtheorem{remark}     [lemma]{Remark}
\newtheorem*{remark*}          {Remark}

\newtheorem*{question*}        {Question}

\crefname{lemma}{Lemma}{Lemma}
\crefname{claim}{Claim}{Claim}
\crefname{theorem}{Theorem}{Theorem}
\crefname{corollary}{Corollary}{Corollary}
\crefname{fact}{Fact}{Fact}
\crefname{proposition}{Proposition}{Proposition}
\crefname{definition}{definition}{definition}
\crefname{remark}{remark}{remark}
\crefname{question}{question}{question}
\crefname{example}{example}{example}
\crefname{section}{Section}{Section}
\crefformat{equation}{(#2#1#3)}
\crefformat{enumi}{(#2#1#3)}

\def\cftry#1#2#3{\expandafter\def\csname #1#3\endcsname{{\csname #2\endcsname{#3}}}}

\def\cfH#1{\ifx#1\cfH\else\cftry{H}{mathbb}#1\expandafter\cfH\fi}
\cfH ABCDEFGHIJKLMNOPQRSTUVWXYZ\cfH
\def\cftd#1{\ifx#1\cftd\else\cftry{td}{tilde}#1\expandafter\cftd\fi}
\cftd ABCDEFGHIJKLMNOPQRSTUVWXYZabcdefghijklmnopqrstuvwxyz\cftd
\def\cfrsd#1{\ifx#1\cfrsd\else\cftry{rsd}{bar}#1\expandafter\cfrsd\fi}
\cfrsd ABCDEFGHIJKLMNOPQRSTUVWXYZabcdefghijklmnopqrstuvwxyz\cfrsd
\def\cfcl#1{\ifx#1\cfcl\else\cftry{cl}{mathcal}#1\expandafter\cfcl\fi}
\cfcl ABCDEFGHIJKLMNOPQRSTUVWXYZabcdefghijklmnopqrstuvwxyz\cfcl
\def\cfkk#1{\ifx#1\cfkk\else\cftry{k}{mathfrak}#1\expandafter\cfkk\fi}
\cfkk ABCDEFGHIJKLMNOPQRSTUVWXYZabcdefghijklmnopqrstuvwxyz\cfkk
\def\cfht#1{\ifx#1\cfht\else\cftry{ht}{hat}#1\expandafter\cfht\fi}
\cfht ABCDEFGHIJKLMNOPQRSTUVWXYZabcdefghijklmnopqrstuvwxyz\cfht
\def\cful#1{\ifx#1\cful\else\cftry{ul}{underline}#1\expandafter\cful\fi}
\cful ABCDEFGHIJKLMNOPQRSTUVWXYZabcdefghijklmnopqrstuvwxyz\cful
\def\cfbf#1{\ifx#1\cfbf\else\cftry{bf}{mathbf}#1\expandafter\cfbf\fi}
\cfbf ABCDEFGHIJKLMNOPQRSTUVWXYZabcdefghijklmnopqrstuvwxyz\cfbf

\newcommand{\df}{\coloneqq}

\DeclareMathOperator{\End}{End}

\DeclareMathOperator{\GL}{GL}

\makeatletter\let\DH\@undefined\makeatother 
\DeclareMathOperator{\DH}{DH}

\DeclareMathOperator{\Sym}{Sym}

\newcommand{\trans}[1]{#1^{\mathrm T}}

\DeclareMathOperator{\diag}{diag}

\let\tmpdet\det\renewcommand{\det}{\tmpdet\nolimits} 
\let\tmpcirc\circ\renewcommand{\circ}{\mathop{\tmpcirc}}

\makeatletter
\def\ifempty#1{\def\@temp{#1}\ifx\@temp\@empty} 
\makeatother
\newcommand{\ifnonempty}[2]{\ifempty{#1}\else#2\fi}
\newcommand{\enclspacing}{}

\newcommand{\cenclose}[7][auto]{%
\ifempty{#1} %
 \ifnonempty{#2}{#2\enclspacing}#3 %
 \ifnonempty{#4}{#7#4#7}#5%
 \ifnonempty{#6}{\enclspacing#6}%
\else\ifthenelse{\equal{#1}{auto}}{%
 \ifthenelse{\equal{#2}{}}{\left.}{\left#2}\enclspacing#3%
 \ifthenelse{\equal{#4}{}}{}{#7\middle #4#7}#5%
 \enclspacing\ifthenelse{\equal{#6}{}}{\right.}{\right#6}
}{
 \ifthenelse{\equal{#2}{}}{}{\csname#1l\endcsname#2}\enclspacing#3%
 \ifthenelse{\equal{#4}{}}{}{#7\csname#1\endcsname#4#7}#5%
 \enclspacing\ifthenelse{\equal{#6}{}}{}{\csname#1r\endcsname#6}
}\fi}
\newcommand{\set}[2][auto]{\enclose[#1]\{{#2}\}}

\newcommand{\abs}[2][auto]{\enclose[#1]|{#2}|}

\newcommand{\of} [2][auto]{\enclose[#1]({#2})}

\newcommand{\enclose}[4][auto]{\cenclose[#1]{#2}{#3}{}{}{#4}{}}

\newcommand{\cset}[3][auto]{\cenclose[#1]\{{#2}\lvert{#3}\}\:}

\newcommand{\dotter}[3][]{#2#3\ifthenelse{\equal{#1}{}}{}{\widehat{#1}#3}#2}
\newcommand{\pts}[2][]{\dotter[#1]{#2}{\cdots}}
\newcommand{\dts}[2][]{\dotter[#1]{#2}{\ldots}}

\newcommand{\mx}[2][r]{\begin{pmatrix*}[#1]#2\end{pmatrix*}}

\newcommand{\longdashrightarrow}{\mathrel{
\longrightarrow\kern-16pt%
{\color{white}\rule[1.6pt]{1.6pt}{1.6pt}}\kern2pt%
{\color{white}\rule[1.6pt]{1.6pt}{1.6pt}}\kern2pt%
{\color{white}\rule[1.6pt]{1.6pt}{1.6pt}}\kern6pt}}
\newcommand{\cmmt}[3][\circlearrowleft]{\ar@{}[#2]|#3{#1}}
\newcommand{\prar}{\ar@{->>}}
\newcommand{\dtar}{\ar@{..>}}
\newcommand{\dshar}{\ar@{-->}}
\newcommand{\nums}[2][]{\ifempty{#1}[#2]\else[#1,#2]\fi}

\let\module\HV 
\DeclareMathOperator{\esMat}{MESP}
\DeclareMathOperator{\esVec}{ESP}
\DeclareMathOperator{\weight}{wt}

\newcommand{\amsfix}{%
\let\qedold\qedsymbol%
\renewcommand{\qedsymbol}{$\Box$}\qed%
\renewcommand{\qedsymbol}{\qedold}}

\title{The Stabilizer of Elementary Symmetric Polynomials}
\author[J.~Hüttenhain]{Jesko Hüttenhain}
\address{
  Technische Universität Berlin\\
  Straße des 17. Juni 136\\
  10623 Berlin\\
  Germany
}
\email[J.~Hüttenhain]{ jesko@math.tu-berlin.de }
\keywords{geometric complexity theory, elementary symmetric polynomials}

\begin{document}

\begin{abstract}
We study the $r$-th elementary symmetric polynomial in $n$ variables with $2<r<n$. There are two kinds of linear transformations on the parameter space that leave this polynomial invariant: Namely, any permutation of the variables and simultaneous scaling by any $r$-th root of unity. We prove that there are no other linear transformations with this property. 
\end{abstract}

\maketitle \vspace{-3.5ex}

\section{Introduction} 

\subsection{Outline} Let $n\in\HN$ and write $\nums n \df \set{1\dts,n}$. The number $n$ will be fixed throughout the document and does not enter notation. We work over the field~$\HC$ of complex numbers throughout. The $r$-th elementary symmetric polynomial for $1\le r\le n$ is defined as
\[
\esVec_r(X_1\dts,X_n) = \sum\nolimits_{\substack{I\subseteq\nums{n}\\\abs I = r}} \prod_{i\in I} X_i \in \HC[X_1\dts,X_n]_r .
\]
Here, $\HC[X_1\dts,X_n]_r$ is the $\HC$-vector space of homogeneous degree $r$ polynomials in $n$ variables.
The elementary symmetric polynomials have been studied in many contexts and provide good training grounds for questions in algebraic complexity theory, especially since much is known about their computational complexity already \cite{STRASSEN197521,ESYM-LOWERBOUND,ESYM-LOWERBOUND-2}. We are interested in geometric complexity theory \cite{MulSoh01} here, or GCT for short. See \cite[pp~28-29]{La14} for a survey on how the elementary symmetric polynomials fit in the context of GCT. For studying polynomial families in the GCT framework, a primary object of interest are their stabilizer groups. Quite surprisingly, it has not been documented for the elementary symmetric polynomials yet -- the primary goal of this note is to fill this gap. As one application, we also study the weights that occur in the coordinate ring of the orbit closure of $\esVec_r$. 

Consider the group $\GL_n\df\GL(\HC^n)$ acting on $\HC^n$ and thereby canonically on the polynomial ring $\HC[X_1\dts,X_n]=\Sym(\HC^n)^\ast$. We are interested in the stabilizer
\[
H_r \df \cset{ h\in \GL_n }{ \esVec_r \circ h = \esVec_r } \subseteq \GL_n
\]
of $\esVec_r$. Our main result is the following statement, which tells us that the stabilizer contains no more than the obvious symmetries:
\begin{theorem}
\label{stab}
Let $2< r< n$. Then, the stabilizer $H_r$ of the $r$-th elementary symmetric polynomial satisfies
$H_r\cong\kS_n\rtimes\HZ_r$, where $\kS_n$ corresponds to the permutation matrices and $\HZ_r=\cset{\omega\in\HC}{\omega^r=1}$ corresponds to the scalar matrices corresponding to some $r$-th root of unity.
\end{theorem} 

\subsection{The Matrix Version} For a matrix $x\in\HC^{n\times n}$, denote by $\esMat_r(x)$ the \mbox{$r$-th} elementary symmetric polynomial in the eigenvalues of~$x$. Note that this is a polynomial function in~$x$, because $\esMat_{n-r}(x)$ is (up to sign) the~$r$-th coefficient of the characteristic polynomial of $x$. From the very late 1950's to the late 1970's, several papers were published to classify those linear transformations in $\End(\HC^{n\times n})$ which leave these coefficients invariant. In \cite{case4} (for the cases $4\le r\le n-1$) and in \cite{case3} (for the case $r=3$), the following main classification is proved:

\begin{theorem}[Marcus-Purves-Beasley] \label{mpb} Let $2<r<n$ and $g\in\End(\HC^{n\times n})$. If $\esMat_r\circ g=\esMat_r$, then there exist $u\in\GL_n$ and $\omega\in\HC$ with $\omega^r=1$ such that one of the following holds:
\begin{itemize}[nolistsep,leftmargin=4ex]
\item $\forall x\in\HC^{n\times n}\colon~g(x) = \omega\cdot uxu^{-1}$.
\item $\forall x\in\HC^{n\times n}\colon~g(x) = \omega\cdot u\trans x u^{-1}$.
\qed 
\end{itemize} 
\end{theorem}

\begin{remark} \label{induced}
\Cref{stab} can now also be interpreted as stating that every element in the stabilizer of $\esVec_r$ is induced by an element in the stabilizer of $\esMat_r$: The action of $H_r$ on $\HC^n$ corresponds to the action of $H_r$ on the space of diagonal matrices via conjugation. Hence, an element $u\in H_r$ corresponds to a map $x\mapsto uxu^{-1}$ leaving diagonal matrices invariant. 
\end{remark}

Unfortunately, one can not deduce \cref{stab} directly from \cref{mpb}. However, the proof of \cref{stab} we give here is based strongly on the ideas from \cite{case4} which were used to prove (part of) \cref{mpb}.

\subsection{Acknowledgement} I am grateful to Peter Bürgisser for bringing this problem to my attention and for his many helpful comments.

\section{The Stabilizer of Elementary Symmetric Polynomials}

For $a\in\HC^n$, we define $\rho(a)$ to be the number of nonzero entries of $a$.
Intuitively, if we imagine~$a$ as a diagonal matrix, $\rho(a)$ is the rank of $a$. Note that \cref{stab:ranklemma,stab:rank1preserved} are analogons of \cite[Lemma~3.3,~3.5]{case4} with very similar proofs.

\begin{lemma} \label{stab:ranklemma} Let $2<r<n$ and $a,b\in\HC^n$. Define the polynomial 
\[ f_{a,b}\df\esVec_r(Xa+b)\df\esVec_r(Xa_1+b_1\dts,Xa_n+b_n)\in\HC[X] \]
Then, we have 
$\deg(f_{a,b})\le 1$ for all $b\in\HC^n$ if and only if $\rho(a)\le1$.
\end{lemma}
\begin{proof}  The \enquote{if} part is clear because $\esVec_r$ is multilinear. For the converse, observe
\begin{align*}
\esVec_r(aX+b) 
&= \sum_{\substack{I\subseteq\nums{n}\\\abs{I}=r}} \prod_{i\in I} (a_iX + b_i) 
= \sum_{\substack{I\subseteq\nums{n}\\\abs{I}=r}} 
  \sum_{s=0}^r \sum_{\substack{J\subseteq I \\\abs{J}=s}} \prod_{j\in J} b_j \prod_{i\in I\setminus J} a_iX 
\\
&= \sum_{s=0}^r \of{ \sum{\substack{J\subseteq I\subseteq\nums{n}\\\abs{I}=r,\abs{J}=s}}  \prod_{j\in J} b_j \prod_{i\in I\setminus J} a_i} X^{r-s},  
\end{align*}
So for any choice of $b\in\HC^n$ and for any $s\le r-2$, we must have 
\[ \sum_{\substack{J\subseteq I\subseteq\nums{n}\\\abs{I}=r,\abs{J}=s}}  \prod_{j\in J} b_j \prod_{i\in I\setminus J} a_i = 0. \] 

For any subset $J_0\subseteq\nums{n}$ of size $s$, choose the vector~$b$ which satisfies $b_j=1$ for $j\in J_0$ and $b_i=0$ for $i\notin J_0$. Then, we get
\[ \sum_{\substack{J_0\subseteq I\subseteq\nums{n}\\\abs{I}=r}} \prod_{i\in I\setminus J_0} a_i = 0.\] 
For $J_0=\set{n-s+1\dts,n}$, we have $\nums{n}\setminus J_0=\nums{n-s}$ and therefore
\[ 0 = \sum_{\substack{I\subseteq\nums{n}\\\abs{I}=r-s}} \prod_{i\in I} a_i = \esVec_{r-s}(a_1\dts,a_{n-s}). \] 
In general, $\esVec_{r-s}$ vanishes on any tuple of $n-s$ entries chosen from~$a$. We set up some notation for what follows: For any $I=\set{i_1\dts,i_k}\subseteq\nums{n}$ with $i_1\pts<i_k$, we write $a_I$ for the sequence $a_{i_1}\dts,a_{i_k}$. With this notation, we can say that
\[ \esVec_{r-s}(a_I)=0 ~\text{for any}~ I\subseteq\nums{n} 
 ~\text{with}~ \abs I = n-s. \]

For $s=r-2$, this means $\esVec_2(a_I)=0$ for any $I\subseteq\nums n$ with $n-r+2$ elements. If all entries of~$a$ were equal, this would imply $a=0$. Hence, up to permutation of the entries of $a$, we may assume $a_1\ne a_2$.
 Let $J\subseteq\set{3\dts,n}$ be a subset with $\abs{J}=n-r+1$. Then,
\begin{align*}
 0 &= \esVec_2(a_1,a_{J}) = a_1\cdot \esVec_1(a_{J}) + \esVec_2(a_{J}) \\
 0 &= \esVec_2(a_2,a_{J}) = a_2\cdot \esVec_1(a_{J}) + \esVec_2(a_{J}) \\
\intertext{yielding, by subtraction,}
 0 &= (a_1-a_2)\cdot \esVec_1(a_{J})
\end{align*} 
so $0=\esVec_1(a_{J})=\sum_{j\in J} a_j$.

We now finish the proof in the case $r\ge 4$. In this case, let $p>q>2$ be two indices. Since $r\ge 4$, we can choose two sets $J_q,J_p\subseteq\set{3\dts,n}$ of size $n-r+1$ such that $q\in J_q$, $p\in J_p$ and $J_q\setminus\set{q}=J_p\setminus\set{p}$. Then,
\begin{align*}
\sum\nolimits_{i\in J_q} a_i &= 0, \\
\sum\nolimits_{i\in J_p} a_i &= 0.
\end{align*}
Subtracting both equalities yields $a_q=a_p$. Since~$p$ and~$q$ were arbitrary in $\set{3\dts,n}$, this implies $c\df a_3\pts=a_n$. Hence, $0=\esVec_1(c\dts,c)$, so $a_i=0$ for all $i>3$. We are done in this case because
\[ 0 = \esVec_2(a_1\dts,a_{n-r+2})=\esVec_2(a_1,a_2)=a_1a_2. \]

We are left to treat the case $r=3$. In this case, $J=\nums{3,n}$. Define the sets $K_i\df\nums{n}\setminus\set{i}$. Since $\abs{K_i}=n-1=n-r+2$, we have $\esVec_2(a_{K_i})=0$ for all $i$. Thus, 
\begin{align}
\nonumber \esVec_2(a_1\dts,a_n) &= a_i\cdot\esVec_1(a_{K_i}) + \esVec_2(a_{K_i})
 \\ \label{stab:zerolater} &= a_i\cdot\esVec_1(a_{K_i}). 
\intertext{Summing over $i$, we obtain}
\nonumber n\cdot\esVec_2(a_1\dts,a_n) 
 &= \sum_{i=1}^n  a_i\cdot\esVec_1(a_{K_i}) = \sum_{i=1}^n \sum_{j\ne i} a_ia_j \\ \nonumber &= 2\cdot\esVec_2(a_1\dts,a_n).
\end{align}
Since $n\ge r\ge 3$, we have $n-2\ne 0$, therefore $\esVec_2(a_1\dts,a_n)=0$. By~\cref{stab:zerolater}, this means that $a_i\cdot\esVec_1(a_{K_i})=0$ for all $i\in\nums n$. Set $\alpha\df\sum_{i=1}^n a_i$, then we have 
\[ a_i\cdot \alpha = a_i \cdot \sum_{i=1}^n a_i = a_i^2 + a_i\cdot \esVec_1(a_{K_i}) = a_i^2. \]
Consequently, $a_i(\alpha-a_i)=0$ for all $i$. Given $a_i\ne0$, we get $a_i=\alpha$, so 
\[ \alpha=\sum_{1\le i\le n} a_i = \sum_{\substack{1\le i\le n\\ a_i\ne 0}} a_i
 = \rho(a)\cdot\alpha,
 \]
which means $\rho(a)=1$ unless $\alpha=0$, in which case $a=0$. 
\end{proof}

\begin{lemma} \label{stab:rank1preserved} Let $2<r<n$.
If $g\in\GL_n$ stabilizes the $r$-th elementary symmetric polynomial and $a\in\HC^n$ satisfies $\rho(a)=1$, then $\rho(g(a))=1$. 
\end{lemma}
\begin{proof} For any $b\in\HC^n$, we have $\deg(f_{a,g^{-1}(b)})\le 1$ by \cref{stab:ranklemma}. Since 
\begin{align*}
 f_{g(a),b} &=\esVec_r(X\cdot g(a) + b)=\esVec_r(g(Xa+g^{-1}(b)))
\\ &=\esVec_r(Xa+g^{-1}(b))=f_{a,g^{-1}(b)},
\end{align*}
we have $\deg(f_{g(a),b})\le 1$ for all $b\in\HC^n$ and again by \cref{stab:ranklemma} this implies $\rho(g(a))\le 1$. Since~$g$ is invertible, $a\ne0$ implies $g(a)\ne0$, so  $\rho(g(a))=1$. 
\end{proof}

\begin{proof}[Proof of \cref{stab}] Let $g\in H_r$, i.e., $g$ stabilizes $\esVec_r$. Let $e_i\in\HC^n$ be the $i$-th canonical base vector, then \cref{stab:rank1preserved} implies that $g(e_i)$ is a vector with only one nonzero entry. In other words, each column of~$g$ contains only one nonzero entry. As~$g$ is invertible, it must be the product of a permutation matrix $\pi$ with an invertible diagonal matrix $t$. We are left to show that~$t$ must be a scalar matrix corresponding to some $r$-th root of unity.

Let $t=\diag(t_1\dts,t_n)$ with $t_i\in\HC^\times$. By assumption, we have
\begin{align*}
\sum_{\substack{I\subseteq\nums n\\\abs I=r}} \prod_{i\in I} X_i &= \esVec_r(X_1\dts,X_n) = \esVec_r(t_1X_1\dts,t_nX_n) 
= \sum_{\substack{I\subseteq\nums n\\\abs I=r}} \prod_{i\in I} t_iX_i
\end{align*}
and comparing coefficients, this means that $t_{i_1}\cdots t_{i_r}=1$ for all sequences of indices $1\le i_1\dts<i_r\le n$. We first use this to show that all the $t_i$ are equal. Let $i,j\in\nums n$ be two distinct indices and arbitrarily chose $r-1$ other indices \mbox{$i_1\dts,i_{r-1}\in\nums n\setminus\set{i,j}$.} Then, the above implies
$\omega\df t_i = \prod_{k=1}^{r-1} t_{i_k}^{-1} = t_j$. This means that~$t$ is a scalar matrix corresponding to $\omega\in\HC^\times$ and we have \mbox{$\omega^r=t_1\cdots t_r=1$,} so $\omega$ is an $r$-th root of unity.
\end{proof}

\begin{remark} Note that $\esVec_r$ is in general not characterized by its stabilizer, i.e., it is not the only homogeneous degree~$r$ polynomial with this stabilizer. To se this, let us assume $n>r>3$ and consider $P\df\esVec_1\cdot\esVec_{r-1}$. We claim that $H_r$ is also the stabilizer of $P$. Assume that $P$ is left invariant by some transformation~$g$. We can write 
\[ \esVec_1\cdot\esVec_{r-1} = P=P\circ g = (\esVec_1\circ g)\cdot(\esVec_{r-1}\circ g). \]
Here, $\esVec_1$ and $\esVec_{r-1}$ are the irreducible factors of $P$ and have distinct degrees, so they are unique up to scalar. More precisely, $\esVec_1\circ g=\alpha\cdot\esVec_1$ and $\esVec_{r-1}\circ g = \alpha^{-1}\cdot\esVec_{r-1}$ for some $\alpha\in\HC^\times$. If $\omega\in\HC^\times$ is such that $\omega^{r-1}=\alpha$, then $\omega g \in H_{r-1}$. By possibly multiplying $\omega$ with an $(r-1)$-st root of unity, \cref{stab} implies that $\omega g$ is a permutation matrix. To see that $g\in H_r$, we need to show that $\omega$ is an $r$-th root of unity. Indeed, $\esVec_1 = \esVec_1\circ\omega g =\omega\alpha\cdot\esVec_1=\omega^r\cdot\esVec_1$.
\end{remark}

\section{Weights of the Coordinate Ring}

Let $V_r\df\HC[X_1\dts,X_n]_r$ be the vector space of homogeneous degree $r$ polynomials in $n$ variables and $\Omega_r \df \esVec_r \circ \GL_n = \cset{ \esVec_r \circ g }{ g\in\GL_n }$ the $\GL_n$-orbit of $\esVec_r$. The Zariski closure $\overline\Omega_r\subseteq V_r$ is a variety on which $\GL_n$ acts from the right. Therefore, the coordinate ring $\HC[\overline\Omega_r]$ decomposes into irreducible \mbox{$\GL_n$-modules}, each corresponding to a certain \emph{dominant weight}. It is a well-known fact that these weights form a semigroup \cite[4.3.5 Theorem]{CHRISDIS}.

We recall several facts about the representation theory of $\GL_n$, see \cite{HUMPHREYS,KRAFT}. The irreducible representations of $\GL_n$ are classified by the semigroup 
\[ \Lambda \df \cset{ \lambda\in\HZ^n }{ \lambda_1\pts\ge\lambda_n } \]
of \emph{dominant weights} and we dentoe by $\module(\lambda)$ the irreducible $\GL_n$-module corresponding to the weight $\lambda\in\Lambda$. Let $S_r\df\cset{\lambda\in\Lambda}{ \HV(\lambda)\subseteq\HC[\overline\Omega_r]}$ be the semigroup of weights that appear in the coordinate ring of the orbit closure of~$\esVec_r$. We will show the following:

\begin{theorem} \label{group}  The group generated by $S_r$ is equal to $\cset{ \lambda\in\Lambda }{ \lambda_1\pts+\lambda_n \in r\HZ }$.
\end{theorem}

\subsection{Representation Theory} We require some more prerequisites from representation theory. A basis of $\module(\lambda)$ is given by the semistandard Young-tableaux~$Y$ of shape $\lambda$. For $\lambda=(4,2,1,1)$ and $n=4$, an example for such a tableaux is
\[ 
\ytableausetup{aligntableaux=center}
Y = \begin{ytableau}
1&2&2&4\\
2&3\\
3 \\
4
\end{ytableau}
\]

\medskip In each row, the numbers are weakly ascending and in each column, they are strictly ascending. Let $\weight(Y)_k\in\HN$ be the number of times that $k\in\HN$ appears in~$Y$. In the above example, $\weight(Y)=(1,3,2,2)$. We then call $\weight(Y)$ the \emph{weight} of~$Y$.
Given a tuple $\alpha=(\alpha_1\dts,\alpha_n)\in\HN^n$ of natural numbers, we call $v\in\module(\lambda)$ a weight vector of weight $\alpha$ if all invertible diagonal matrices $t=\diag(t_1\dts,t_n)\in\GL_n$ act on~$v$ via 
$t.v=t_1^{\alpha_1}\cdots t_n^{\alpha_n}\cdot v$.
Any semistandard Young tableaux~$Y$ is a weight vector of weight $\weight(Y)$. A permutation matrix corresponding to a permutation $\pi\in\kS_n$ will map a weight vector of weight $\alpha=(\alpha_1\dts,\alpha_n)\in\HN^n$ to a weight vector of weight $\pi.\alpha = (\alpha_{\pi(1)}\dts,\alpha_{\pi(n)})$. 


For $\lambda=(\lambda_1\dts,\lambda_n)\in\Lambda$, we set $\lambda^\ast\df(-\lambda_n\dts,-\lambda_1)$. Clearly $\lambda^\ast\in\Lambda$, in fact $\lambda\mapsto\lambda^\ast$ is an automorphism of $\Lambda$. We quote the following, well-known fact about the dual of a representation from \cite[III.1.4, Bemerkung~2]{KRAFT}:

\begin{lemma} \label{StarDual} For any $\lambda\in\Lambda$,  we have $\module(\lambda)^\ast=\module(\lambda^\ast)$.  \qed 
\end{lemma}

\Cref{group} now follows from the following, slightly more general observation together with \cref{stab}.

\begin{proposition} Let~$V$ be a $\GL_n$-module, $2<r<n$ and $v\in V$ any element whose stabilizer is equal to $H_r\subseteq\GL_n$. Let $\Omega\df\GL_n.v$ be its orbit. Then, the weights that appear in the coordinate ring $\HC[\Omega]$ generate the lattice $\cset{\lambda\in\Lambda}{\lambda_1\pts+\lambda_n\in r\HZ}$.
\end{proposition}
\begin{proof} We know by the algebraic Peter-Weyl Theorem  \cite[27.3.9]{TAUVELYU} that there is an isomorphism of $\GL_n\times\GL_n$-modules
\[
\HC[\GL_n]\cong\bigoplus_{\lambda\in\Lambda} \module(\lambda)^\ast\otimes\module(\lambda).
\]
This isomorphism is given by sending a tensor $\varphi\otimes v\in\module(\lambda)^\ast\otimes\module(\lambda)$ to the regular fuction $g\mapsto \varphi(g.v)$, hence the right action on $\HC[\GL_n]$ by $\GL_n$ corresponds to the action on the right tensor factor. There is also a $\GL_n$-module isomorphism $\HC[\GL_n.v]\cong\HC[\GL_n]^{H_r}$  where invariants are taken with respect to the right action \cite[25.4.7 and 25.5.2]{TAUVELYU}. Thus,
\[ \HC[\GL_n.v] \cong \bigoplus_{\lambda\in\Lambda} \module(\lambda)^\ast\otimes\module(\lambda)^{H_r}. \]
Set $\Lambda_r\df\cset{\lambda\in\Lambda}{\lambda_1\pts+\lambda_r\in r\HZ}$. By the above formula and \cref{StarDual}, we want to show that $\Lambda_r$ is generated by all $\lambda^\ast$ with $\module(\lambda)^{H_r}\ne\set0$. Since the idempotent map $\lambda\mapsto\lambda^\ast$ restricts to an automorphism of $\Lambda_r$, it suffices to show that $\Lambda_r$ is equal to the lattice $\tilde\Lambda_r$ generated by all $\lambda\in\Lambda$ with $\module(\lambda)^{H_r}\ne\set0$.

First, we show that $\tilde\Lambda_r\subseteq\Lambda_r$. Assume that $\module(\lambda)^{H_r}\ne\set0$, then any scalar matrix $t\in\HC$ acts on $\module(\lambda)$ by $t^{\lambda_1\pts+\lambda_n}$. Since the scalar matrices that correspond to $r$-th roots of unity are all contained in~$H_r$ and~$H_r$ stabilizes some vector in $\module(\lambda)$, we may conclude that $\lambda_1\pts+\lambda_n$ is divisible by $r$, hence $\lambda\in\Lambda_r$. Consequently, $\tilde\Lambda_r\subseteq\Lambda_r$.

We now prove the other inclusion and first claim that $(r,0\dts,0)\in\tilde\Lambda_r$. The module corresponding to this partition is the $r$-th symmetric power $\Sym^r\HC^n$. Denoting by $e_i\in\HC^n$ the $i$-th canonical base vector, we can see that its $r$-th symmetric power $e_i^r\in\Sym^r\HC^n$ is a nonzero vector which is invariant under all permutations and also invariant under scaling by $r$-th roots of unity, hence $e_i^r$ is $H_r$-invariant.

For the second step, let $\ell_i\df r\cdot\frac{n(n+1)}2 - i$. We claim that for $1\le i \le n-1$, the partitions
\[ \lambda_i\df(\ell_i,\underset{i~\text{times}}{\underbracket{1\dts,1}},0\dts,0) \]
are all contained in $\tilde\Lambda_r$. Indeed, we consider the semistandard Young tableaux of shape $\lambda_i$ that contains the number~$k$ precisely $kr$ many times, filling in the values from left to right, top to bottom. This corresponds to a nonzero weight vector~$w_i$ of weight $(r,2r,3r\dts,nr)=r\cdot(1,2\dts,n)$ which is therefore invariant under $r$-th roots of unity. 
If we apply any permutation $\pi\in\kS_n$, then we obtain a weight vector $\pi.w_i$ of weight $r\cdot(\pi(1)\dts,\pi(n))$ which is still invariant under $r$-th roots of unity. Since these weights are all pairwise distinct, the vectors $\pi.w_i$ are linearly independent and the symmetrization
\[ \rsdw_i\df\sum_{\pi\in\kS_n} \pi.w_i \]
is a nonzero $H_r$-invariant. This proves $\module(\lambda_i)^{H_r}\ne\set0$, therefore $\lambda_i\in\tilde\Lambda_r$. 

We have shown that $\tilde\Lambda_r$ contains any $\HZ$-linear combination of the columns of the following matrix:
\[
\mx[c]{ 
r & \ell_1 & \cdots & \ell_{n-1} \\
0 & 1 & \cdots  & 1 \\
\vdots & \ddots  & \ddots & \vdots  \\
0 & \cdots & 0 &  1 } \]
To show $\Lambda_r\subseteq\tilde\Lambda_r$, we are left to verify that any $\lambda\in\Lambda_r$ is a $\HZ$-linear combination of the columns of $A$. To see this, we subtract apropriate multiples of the last $n-1$ columns from $\lambda$ to eliminate all but the first coordinate. The resulting vector $\mu=(a,0\dts,0)$ is an element of $\tilde\Lambda_r$ because all columns of $A$ are in $\tilde\Lambda_r$. This implies that $a$ must be divisible by $r$ and so $\mu$ is an integer multiple of the first column of $A$. Thus, $\lambda$ is a $\HZ$-linear combination of the columns of $A$.
\end{proof}

\raggedright
\bibliography{esym}{}
\bibliographystyle{alpha}
\end{document}